\documentclass[11pt,a4paper]{article}

\usepackage{fullpage,graphicx,url}
\usepackage{amsmath,amsthm,amssymb}

\newtheorem{thm}{Theorem}
\newtheorem{cor}[thm]{Corollary}

\title{The extremal generalised Randi\'c index for a given degree range}
\author{John Haslegrave\thanks{Lancaster University, UK, email \texttt{j.haslegrave@lancaster.ac.uk}}}

\begin{document}
\maketitle
\begin{abstract}O and Shi proved that the Randi\'c index of any graph $G$ with minimum degree at least $\delta$ and maximum degree at most $\Delta$ is at least $\frac{\sqrt{\delta\Delta}}{\delta+\Delta}|G|$, with equality if and only if the graph is $(\delta,\Delta)$-biregular. In this note we give a short proof via a more general statement. As an application of our more general result, we classify for any given degree range which graphs minimise (or maximise) the generalised Randi\'c index for any exponent, and describe the transitions between different types of behaviour precisely.
\end{abstract}

\section{Introduction}
Topological indices of graphs are numerical invariants that give a measure of network structure. They are widely studied in chemical graph theory owing to their original application to graphs representing organic molecules. Interest in such descriptors began with the study by Wiener in 1947 of the correlation between boiling points of alkanes and their path number \cite{Wie47}, subsequently known as the Wiener index. 

Many other indices have since been proposed, and shown to similarly predict physical quantities of interest. Most prominent among these are the Randi\'c index \cite{Ran75} and the Zagreb indices developed by Gutman, Ru\v{s}\v{c}i\'c, Trinajsti\'c and Wilcox \cite{GT72,GRTW75}, all of which measure the ``branching'' of a molecule and have the advantage of being expressed as a sum of local contributions, which we refer to below as being ``degree-based''; a more general treatment of such indices was given by Do\v{s}li\'c, R\'eti and Vukic\v{e}vic \cite{DRV11}.

The Randi\'c index, introduced in 1975 by Milan Randi\'c \cite{Ran75}, is one of the most-studied topological graph indices: see the survey by Li and Shi \cite{LS08} and references therein. It is defined as \begin{equation}R(G)=\sum_{uv\in E(G)}(d(u)d(v))^{-1/2},\label{randic}\end{equation}
where $d(u)$ is the degree of vertex $u$.

Among connected $n$-vertex graphs, or more generally those with all degrees positive, Bollob\'as and Erd\H{o}s \cite{BE98} showed that the Randi\'c index is minimised by the star, with value $\sqrt{n-1}$. However, for bounded-degree graphs it is $\Theta(n)$, and O and Shi \cite{OS18}, improving on an earlier bound due to Aouchiche, Hansen and Zheng \cite{AHZ07}, gave the best-possible lower bound for graphs with all degrees in a specified range. The maximum value of the Randi\'c index in a given range is less interesting, being attained by any regular graph independent of the degree.

Here we give a short proof of a more general statement that immediately implies the result of O and Shi, but applies to any degree-based topological index (which we define below). We then apply our result to the generalised Randi\'c index $R_\alpha(G)$, introduced by Bollob\'as and Erd\H{o}s \cite{BE98}, where the exponent $-1/2$ in \eqref{randic} is replaced by a fixed real number $\alpha$. We give a complete classification of the minimal value of the generalised Randi\'c index for graphs in any given degree range and for any value of $\alpha$, which is divided into three regimes. We remark that similar, but more complicated, methods were used to analyse two specific degree-based indices by Deng, Balachandran and Elumalai \cite{DBE19}, building on work of Dalf\'o \cite{Dal19}.

\section{Results}
We use standard notation, including $|G|$ for the order of a graph $G$ and $N(u)$ for the set of neighbours of $u$. For integers $r,s\geq 1$, we say a graph is $(r,s)$-biregular if every vertex has degree either $r$ or $s$, and every edge has one vertex of each degree. (Note that $(r,r)$-biregular is synonymous with $r$-regular.) We define a \textit{degree-based} (topological) index to be a graph invariant of the form 
\begin{equation}\label{degree-based}F(G)=\sum_{uv\in E(G)}f(d(u),d(v))\end{equation}
for some function $f$ satisfying $f(x,y)\equiv f(y,x)$. The following result gives sharp bounds on the extremal values for graphs with degrees in a specified range (noting that maximising $F(G)$ is equivalent to minimising $-F(G)$).

\begin{thm}\label{main}Fix integers $\Delta\geq \delta\geq 1$ and let $F$ satisfy \eqref{degree-based}. Let $a,b\in \{\delta,\ldots,\Delta\}$ with $a\leq b$ be chosen to minimise $\frac{ab\,f(a,b)}{a+b}$. Then any graph $G$ with minimum degree at least $\delta$ and maximum degree at most $\Delta$ satisfies $F(G)\geq \frac{abf(a,b)}{a+b}|G|$. Equality occurs if $G$ is $(a,b)$-biregular, and this is the only case when equality occurs provided the choice of $a,b$ is unique.
\end{thm}
\begin{proof}Consider the edge weighting $w:E(G)\to\mathbb R^+$ given by $w(uv)=1/d(u)+1/d(v)$. Note that 
\[\sum_{e\in E(G)}w(e)=\sum_{u\in V(G)}\sum_{v\in N(u)}1/d(u)=\sum_{u\in V(G)}1=|G|.\]
For any edge $uv$, where without loss of generality $d(u)\leq d(v)$, we have 
\begin{equation}\label{edge-compare}\frac{f(d(u),d(v))}{w(uv)}=\frac{d(u)d(v)f(d(u),d(v))}{d(u)+d(v)}\geq\frac{ab\,f(a,b)}{a+b},\end{equation}
with equality if, and provided $a,b$ are uniquely determined only if, $d(u)=a$ and $d(v)=b$.
Thus \[F(G)\geq \sum_{e\in E(G)}\frac{ab\,f(a,b)}{a+b}w(e)=\frac{ab\,f(a,b)}{a+b}|G|,\]
with equality if and only if equality holds in \eqref{edge-compare} for every edge.
\end{proof}
For $f(a,b)=(ab)^{-1/2}$, corresponding to the Randi\'c index, we have 
\[\frac{ab\,f(a,b)}{a+b}=(x+1/x)^{-1},\]
where $x=\sqrt{b/a}$. Since $x+1/x$ is increasing for $x\geq 1$, this is minimised (uniquely) when $(a,b)=(\delta,\Delta)$ and maximised when $a=b$. Thus we immediately obtain the following corollary.
\begin{cor}If $G$ has minimum degree at least $\delta$ and maximum degree at most $\Delta$, then 
	\[\frac{\sqrt{\delta\Delta}}{\delta+\Delta}|G|\leq R(G)\leq\frac{1}{2}|G|,\]
	with the lower bound being achieved if and only if $G$ is $(\delta,\Delta)$-biregular and the upper bound being achieved if and only if every component of $G$ is regular.
\end{cor}
We conclude by applying Theorem \ref{main} to the generalised Randi\'c index $R_\alpha(G)$.

\begin{figure}[ht]
	\centering\includegraphics{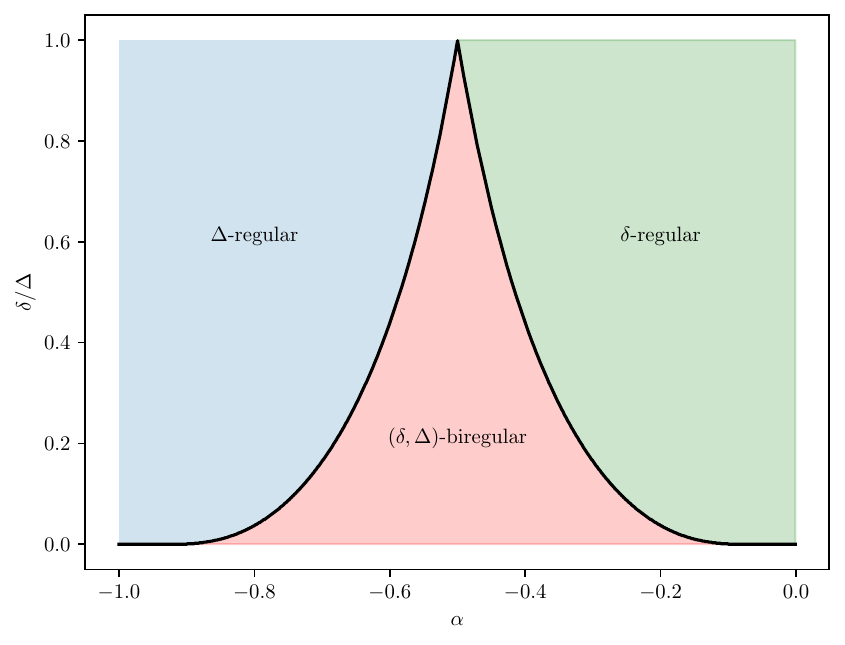}
	
	\caption{The three regions having different minimising graphs for the generalised Randi\'c index.}
\end{figure}
\begin{thm}
	Fix integers $\Delta>\delta>0$ and write $c=\Delta/\delta$. If $G$ has minimum degree at least $\delta$ and maximum degree at most $\Delta$, then:
	\begin{enumerate}
		\item for $\alpha\leq\log_c\Bigl(\frac{1+c}{2c}\Bigr)$, we have $R_\alpha(G)\geq\frac{\Delta^{1+2\alpha}}{2}|G|$, with equality if $G$ is $\Delta$-regular;
		\item for $\log_c\Bigl(\frac{1+c}{2c}\Bigr)\leq\alpha\leq \log_c\Bigl(\frac2{1+c}\Bigr)$, we have $R_\alpha(G)\geq \frac{(\delta\Delta)^{1+\alpha}}{\delta+\Delta}|G|$, with equality if $G$ is $(\delta,\Delta)$-biregular;
		\item for $\log_c\Bigl(\frac2{1+c}\Bigr)\leq \alpha$, we have $R_\alpha(G)\geq \frac{\delta^{1+2\alpha}}{2}|G|$, with equality if $G$ is $\delta$-regular.
	\end{enumerate}
	These are the only graphs where equality occurs unless $\alpha$ lies on the boundary between two regimes, in which case both equality cases apply.
	
	Additionally, for $\alpha<-1/2$ we have $R_\alpha(G)\leq \frac{\delta^{1+2\alpha}}{2}|G|$, with equality if and only if $G$ is $\delta$-regular, whereas for $\alpha>-1/2$ we have $R_\alpha(G)\leq \frac{\Delta^{1+2\alpha}}{2}|G|$, with equality if and only if $G$ is $\Delta$-regular.
\end{thm}
\begin{proof}
	Consider optimising the function $g(x,y)=\frac{x+y}{(xy)^{1+\alpha}}$ for  $x,y\in[\delta,\Delta]$, and write $c=\delta/\Delta$. Note that
	\[\frac{\partial g(x,y)}{\partial x}=\frac{-\alpha x-(1+\alpha)y}{x^{2+\alpha}y^{1+\alpha}}\quad\text{and}\quad\frac{\partial^2 g(x,y)}{\partial x^2}=\frac{\alpha(1+\alpha) x+(1+\alpha)(2+\alpha)y}{x^{3+\alpha}y^{1+\alpha}},\]
	and similarly for $y$.
	
	For $\alpha\leq -1$, $g$ is increasing in both coordinates and so is maximised for $x=y=\Delta$. Similarly, $g$ is maximised for $x=y=\delta$ if $\alpha\geq 0$.
	
	We now consider the case $-1\leq\alpha\leq-1/2$. Here $g$ is an increasing function of the larger coordinate, and a convex function of the smaller. It follows that the maximum occurs only when the larger coordinate is $\Delta$ and the smaller is either $\delta$ or $\Delta$. Comparing $g(\delta,\Delta)$ with $g(\Delta,\Delta)$, we see that the former is larger when $1+c\geq 2c^{1+\alpha}$, which is equivalent to $\alpha\geq\log_c\Bigl(\frac{1+c}{2c}\Bigr)$.
	
	The case $-1/2\leq\alpha\leq 0$ is similar, with inequalities reversed and $\delta,\Delta$ exchanged. Thus we compare $g(\delta,\Delta)$ with $g(\delta,\delta)$, with the former being larger if $1+c\geq 2c^{-\alpha}$, which is equivalent to $\alpha\leq \log_c\Bigl(\frac2{1+c}\Bigr)$.
	
	By Theorem \ref{main}, $R_\alpha(G)\geq M^{-1}|G|$, where $M$ is the maximum value of $g(x,y)$, and this, together with consideration of equality cases, gives the claimed lower bounds.
	
	For the upper bounds, we instead minimise $g$. Here, for any $\alpha<-1/2$ we have $g$ increasing in $x$ for $x\geq\frac{-1-\alpha}{\alpha}y$, and so from any $(x,y)\neq (\delta,\delta)$ the value of $g(x,y)$ can be strictly decreased by reducing one coordinate. Thus the minimum uniquely occurs at $g(\delta,\delta)$, and similarly for $\alpha>-1/2$ it uniquely occurs at $g(\Delta,\Delta)$.
\end{proof}

\section*{Acknowledgments}
I am grateful to Saieed Akbari for bringing the result of \cite{OS18} to my attention.

\end{document}